\documentclass[runningheads]{llncs}
\usepackage[utf8]{inputenc}
\usepackage[english]{babel}
\usepackage{amsmath}
\usepackage{amsfonts}
\usepackage{amssymb}
\usepackage{graphicx}
\usepackage{color}

\newtheorem{prop}{Proposition}
\newtheorem{cor}{Corollary}
\newtheorem{thm}{Theorem}
\newtheorem{rem}{Remark}
\newtheorem{lem}{Lemma}

\title{Computing Nonlinear Eigenfunctions via Gradient Flow Extinction}
\author{Leon Bungert \and Martin Burger \and Daniel Tenbrinck}
\begin{document}
\authorrunning{L. Bungert et al.}
%
\institute{Department Mathematik, Universit\"{a}t Erlangen-N\"{u}rnberg, Cauerstrasse 11, 91058 Erlangen, Germany.\\
\email{\{leon.bungert,martin.burger,daniel.tenbrinck\}@fau.de}}
\maketitle              
\begin{abstract}
In this work we investigate the computation of nonlinear eigenfunctions via the extinction profiles of gradient flows. We analyze a scheme that recursively subtracts such eigenfunctions from given data and show that this procedure yields a decomposition of the data into eigenfunctions in some cases as the 1-dimensional total variation, for instance. We discuss results of numerical experiments in which we use extinction profiles and the gradient flow for the task of spectral graph clustering as used, e.g., in machine learning applications. 

\end{abstract}

\section{Introduction}
Linear eigenvalue problems are of utter importance and a classical tool in signal and image processing. A frequently used tool here is the Fourier transform which basically decomposes a given signal into eigenfunctions of the Laplacian operator and makes frequency-based filtering possible. In addition, such problems also find their applications in machine learning and the treatment of large data sets~\cite{shi2000normalized,ng2002spectral}. However, for some applications -- as for example certain graph clustering tasks -- linear theory does not suffice to achieve satisfactory results. Therefore, nonlinear eigenproblems, which involve a nonlinear operator, have gained in popularity over the last years since they can be successfully applied in far more complex and interesting application scenarios. However, solving such nonlinear eigenproblems is a challenging task and the techniques heavily depend on the structure of the involved operator. The setting we adopt is the following: we consider a Hilbert space $\mathcal H$ and study the eigenvalue problem related to the subdifferential $\partial J$ of an absolutely one-homogeneous convex functional $J:\mathcal H\to\mathbb R\cup\{+\infty\}$. A prototypical example for such a functional is the total variation. So called \emph{nonlinear eigenfunctions} are characterized by the inclusion
\begin{align}\label{eigenfunction}
\lambda p\in\partial J(p),
\end{align}
where usually the normalization $\|p\|=1$ is demanded to have a interpretable eigenvalue $\lambda$. Note that the operator $\partial J$ is nonlinear and multivalued, in particular eigenfunctions do not form linear subspaces. Important properties and characterizations of the nonlinear eigenfunctions are collected in \cite{bungert2019nonlinear}. Of particular interest in applications is the decomposition of some data $f\in\mathcal{H}$ into a linear combination of eigenfunctions which, for instance, allows for scale-based filtering in image processing \cite{burger2015spectral,gilboa2014total,gilboa2018nonlinear} -- analogously to linear Fourier methods. In other applications, as spectral graph clustering, one is rather interested in finding a specific eigenfunction that is in some way related to the data or captures topological properties of the domain. An important tool for this nonlinear spectral analysis is the so called \emph{gradient flow} of the functional $J$ 
\begin{align}\label{gradflow}\tag{GF}
\begin{cases}
u'(t)=-p(t),\quad p(t)\in\partial J(u(t)),\\
u(0)=f,
\end{cases}
\end{align}
whose connection to  which has been analysed in finite dimensions in \cite{burger2016spectral} and in infinite dimensions in \cite{bungert2019nonlinear}. In particular, the authors proved that the gradient flow is able to achieve the above-mentioned decomposition task in some situations. Furthermore, it always generates one specific eigenfunction, called the \emph{extinction profile} or \emph{asymptotic profile} of $f$ (cf.~\cite{andreu2002some} for the special case of total variation flow). This profile is given by the subsequential limit of $p(t)$ as $t$ tends to the extinction time of the flow. A different flow which also generates an eigenfunction was introduced and analyzed in \cite{aujol2017theoretical,nossek2018flows}. A third way for obtaining eigenfunctions, being less rigorous and reliable, consists in computing the gradient flow \eqref{gradflow} of $f$ and checking for subgradients $p(t)$ to be eigenfunctions.

The rest of this work is organized as follows: After recapping some notation and important results regarding gradient flows and associated eigenfunctions in Section~\ref{sec:gf_eigenfunctions}, we analyze an iterative scheme in Section~\ref{sec:scheme} which is based on extinction profiles and constitutes an alternative to the already existent decomposition into nonlinear eigenfunctions through subgradients of the gradient flow. Finally, in Section~\ref{sec:applications} we present some applications of extinction profiles, mainly to spectral clustering.

\section{Gradient Flows and Eigenfunctions}
\label{sec:gf_eigenfunctions}
Without loss of generality, we will assume that the data $f$ is orthogonal to the null-space of the functional $J$ which is denoted by $\mathcal{N}(J)$. For the example of total variation, this corresponds to calculating with data functions of zero mean. Furthermore, we will only be confronted with eigenvectors of eigenvalue $1$, which follows naturally from the gradient flow structure. Normalizing them to have unit norm, shows that a suitable eigenvalue for an element $p$ which meets $p\in\partial J(p)$ is given by $\|p\|$. A complete picture of the theory of gradient flows and nonlinear eigenproblems is given in \cite{bungert2019nonlinear}, from where the following statements are taken.

An important property of the gradient flow \eqref{gradflow} is that it decomposes the data $f$ into subgradients of the functional $J$, i.e., it holds
\begin{align}\label{reconstruction}
f=\int_0^\infty p(s)\,ds,
\end{align}
where the subgradients $p(s)$ enjoy the regularity of being elements in $\mathcal{H}$ and, furthermore, have minimal norm in the subdifferentials $\partial J(u(s))$, i.e. $\|p(s)\|\leq\|p\|$ for all $p\in\partial J(u(s))$. This naturally qualifies them for being eigenfunctions as it was shown in \cite{bungert2019nonlinear}. Furthermore, the solution $u$ of \eqref{gradflow} extincts to zero in finite time under generic conditions on the functional $J$. More precisely, it has to satisfy a Poincar\'{e}-type inequality, namely that there is $C>0$ such that
\begin{align}\label{ineq:poincare}
\|u\|\leq CJ(u)
\end{align}
holds for all $u$ which are orthogonal to the null-space of $J$ (cf.~\cite[Rem.~6.3]{bungert2019nonlinear}). Let us in the following assume that the solution $u$ of \eqref{gradflow} extincts at time $0<T<\infty$, in other words $u(t)=0$ for all $t\geq T$. In that case, there is an increasing sequence of times $(t_n)$ converging to $T$ such that
\begin{align}\label{extinction_profile}
p^*:=\lim_{n\to\infty}\frac{1}{T-t_n}\int_{t_n}^Tp(s)\,ds
\end{align}
is a non-trivial eigenfunction of $\partial J$, i.e., $p^*\neq 0$ and $p^*\in\partial J(p^*)$. The element $p^*$ is referered to as an {extinction profile} of $f$. 

In the following, the term \emph{spectral case} refers to the scenario that the subgradients $p(t)$ in \eqref{gradflow} are eigenfunctions themselves, i.e., $p(t)\in\partial J(p(t))$ for all $t>0$. Using this together with the fact that $\|p(t)\|$ is decreasing in $t$ implies that \eqref{reconstruction} becomes a decomposition of the datum into eigenfunctions with decreasing eigenvalues. Several scenarios and geometric conditions for this to happen were investigated in \cite{bungert2019nonlinear}. For instance, if the functional $J$ is the total variation in one space dimension, a divergence and rotation sparsity term, or special finite dimensional $\ell^1$-sparsity terms, one has this spectral case. Also for general $J$, a special structure of the data (cf.~\cite{bungert2018solution,bungert2019nonlinear,schmidt2018inverse}) can yield the spectral case. Let us conclude the nomenclature by introducing the quantity
\begin{align}\label{dual_norm}
\|f\|_*=\sup\left\lbrace\langle f,p\rangle\,:\,J(p)=1,\;p\in\mathcal N(J)^\perp\right\rbrace,\quad f\in\mathcal H,
\end{align}
and noting that \eqref{ineq:poincare} implies 
\begin{align}\label{ineq:imbedding}
\|f\|_*\leq C\|f\|,\quad\forall f\in\mathcal H.
\end{align}

\section{An Iterative Scheme to Compute Nonlinear Spectral Decompositions}
\label{sec:scheme}
If one is not in the above-explained spectral case, the decomposition of an arbitrary data into nonlinear eigenfunctions is a hard task. A very intuitive approach into this direction is to compute an eigenfunction, subtract it from the data, and start again. As already mentioned there are several approaches to get hold of a nonlinear eigenfunction, one of which consists in the computation of extinction profiles. We will use these to define and analyze a recursive scheme for the decomposition of data into nonlinear eigenfunctions. We consider
\begin{align}\label{scheme}\tag{S}
\begin{cases}
f_0&:=f,\\
f_{n+1}&:=f_{n}-c_np^*_n,\quad n\geq0,
\end{cases}
\end{align}
where $p_n^*$ denotes the extinction profile of $f_n$ and $c_n:=\langle f_n,p^*_n\rangle/\|p^*_n\|^2$. Note that despite being explicit, the scheme still requires the non-trivial computation of the extinction profiles $p_n^*$ of $f_n$. A numerical approach for this subprocedure is given in Section~\ref{sec:numerics}. The scheme can be rewritten as
\begin{align}\label{scheme2}\tag{S'}
f_{n+1}=f-\sum_{i=0}^nc_ip^*_i,\quad n\geq0.
\end{align}
Hence, if there is $N\in\mathbb N$ such that $f_{N+1}=0$, it holds $f=\sum_{i=0}^Nc_ip_i^*,$ which means that $f$ can be written as linear combination of finitely many eigenfunctions of $\partial J$. More generally, if there is some $g\in\mathcal H$ such that $f_n\to g$ as $n\to\infty$, one has $f=g+\sum_{i=0}^\infty c_ip_i^*,$ which corresponds to the decomposition of $f$ into a linear combination of countably many eigenfunctions and a rest $g$. 
 
Let us start by collecting some essential properties of the iterative scheme~\eqref{scheme}.
\begin{prop}
One has the following statements:
\begin{enumerate}
\item The scheme \eqref{scheme} terminates (i.e. $f_{n+1}=f_n$) if and only if $p^*_n$ is orthogonal to $f_n$. In the spectral case, this happens if and only if $f_n=0$.
\item $\|f_{n}\|$ is strictly decreasing at a maximal rate until termination.
\end{enumerate}
In more detail, it holds
\begin{align}\label{eq:norms}
\|f_{n+1}\|^2=\|f_n\|^2-\frac{\langle f_n,p_n^*\rangle^2}{\|p_n^*\|^2},\quad n\geq0.
\end{align}
\end{prop}
\begin{proof}
Ad 1.: Obviously the scheme terminates if $c_n=0$ which is the case if and only if $\langle f_n,p_n^*\rangle=0$. In the spectral case, $c_n$ equals the extinction time of the gradient flow with initial data $f_n$ (cf.~\cite{bungert2019nonlinear}). Due to continuity of the solution of the gradient flow, this is zero if and only if $f_n=0$. 

Ad 2.: One has $\|f_n-cp_n^*\|^2=\|f_n\|^2-\left[2c\langle f_n,p_n^*\rangle-c^2\|p_n^*\|^2\right]$ for any $n\geq0$ and $c\in\mathbb{R}$. The term in square brackets is quadratic in $c$ and zero for $c\in\left\{0,2{\langle f_n,p_n^*\rangle}/{\|p_n^*\|^2}\right\}$. Hence, it is maximal for $c=c_n=\langle f_n,p_n^*\rangle/\|p_n\|^2$. This concludes the proof.
\end{proof}

\begin{rem}[Well-definedness]
Note that the scheme \eqref{scheme} is well-defined due to the Poincar\'{e} inequality \eqref{ineq:poincare}. It can be used to bound the extinction time $T(f)$ of the gradient flow with datum $f$. In more detail, it holds $T(f)\leq C\|f\|$, as was shown in \cite{bungert2019nonlinear}. Consequently, the flows with datum $f_n$ for $n\geq 1$  all have finite extinction time since, $T(f_n)\leq C\|f_n\|\leq C\|f\|<\infty$.
\end{rem}

We start with a Lemma that will be useful for proving convergence of \eqref{scheme} in the spectral case.
\begin{lem}\label{lem:decrease}
It holds $\sum_{n=0}^\infty\frac{\langle f_n,p_n^*\rangle^2}{\|p_n^*\|^2}<\infty$ and, in particular, $\lim_{n\to\infty}\frac{\langle f_n,p_n\rangle}{\|p_n^*\|}=0.$
\end{lem}
\begin{proof}
Summing \eqref{eq:norms} for $n=0,\dots,K$ and using $f_0=f$ yields
$$\|f\|^2=\|f_{K+1}\|^2+\sum_{n=0}^K\frac{\langle f_n,p_n^*\rangle^2}{\|p_n^*\|^2}\geq \sum_{n=0}^K\frac{\langle f_n,p_n^*\rangle^2}{\|p_n^*\|^2}.$$
Letting $K$ tend to infinity concludes the proof.
\end{proof}

\begin{cor}
It holds $\sum_{n=0}^\infty\|f_{n+1}-f_n\|^2=\sum_{n=0}^\infty\frac{\langle f_n,p_n^*\rangle^2}{\|p_n^*\|^2}<\infty.$
\end{cor}

\subsection{The Spectral Case}
In the spectral case, one already has the decomposition of $f$ as integral over $p(t)$ for $t>0$, where $p(t)$ is given by \eqref{gradflow}. Still, we study scheme \eqref{scheme} and prove that it provides an alternative decomposition into a discrete sum of eigenfunctions.
\begin{thm}
In the spectral case the sequence $(f_n)$ generated by \eqref{scheme} converges weakly to~$0$.
\end{thm}
\begin{proof}
It holds according to \cite{bungert2019nonlinear} that $\|f_n\|_*={\langle f_n,p_n^*\rangle}/{J(p_n^*)}$. Using $J(p_n^*)=\|p_n^*\|^2$ we find
$$\lim_{n\to\infty}\|f_n\|_*\|p_n^*\|=\lim_{n\to\infty}\frac{\langle f_n,p_n\rangle}{\|p_n^*\|}=0$$
by Lem.~\ref{lem:decrease}. However, $\|p_n^*\|$ cannot be a null sequence since \eqref{ineq:imbedding} implies $1=\|p_n^*\|_*\leq C\|p_n^*\|^2.$ Therefore, $\|f_n\|_*$ converges to zero which immediately implies that $f_n\rightharpoonup 0$ in~$\mathcal H$.
\end{proof}

\begin{cor}[Parseval identity]
In the spectral case one has $f=\sum_{n=0}^\infty c_np_n^*$, where the equality is to be understood in the weak sense, i.e., tested against any $v\in\mathcal H$. In particular, for $v=f$ this implies the \emph{Parseval identity} 
$$\|f\|^2=\sum_{i=0}^\infty c_i\langle p_i^*,f\rangle,$$
which is a perfect analogy to the linear Fourier transform.
\end{cor}

\subsection{The General Case}
In the general case, one only has much weaker statements about the iterates of the scheme \eqref{scheme}. Indeed, one can only prove weak convergence of a subsequence of $(f_n)$. Furthermore, the limit is not zero, in general, meaning that there remains a rest which cannot be decomposed into eigenfunctions by the scheme. 

\begin{thm}
The sequence $(f_n)$ generated by \eqref{scheme} admits a subsequence $(f_{n_k})$ that converges weakly to some $g\in\mathcal H$.
\end{thm}
\begin{proof}
By \eqref{eq:norms}, the sequence $(f_n)$ is bounded in $\mathcal{H}$ and, therefore, admits a convergent subsequence. 
\end{proof}

To conclude this section, we mention that scheme \eqref{scheme} provides an alternative decomposition into eigenfunctions in the spectral case. In the general case, however, the scheme is of limited use since it fails to decompose the whole datum, in general. Furthermore, the indecomposable rest can still contain a large amount of information as Figure~\ref{fig:barb} shows. 

\begin{figure}
\centering
\includegraphics[width=0.3\textwidth]{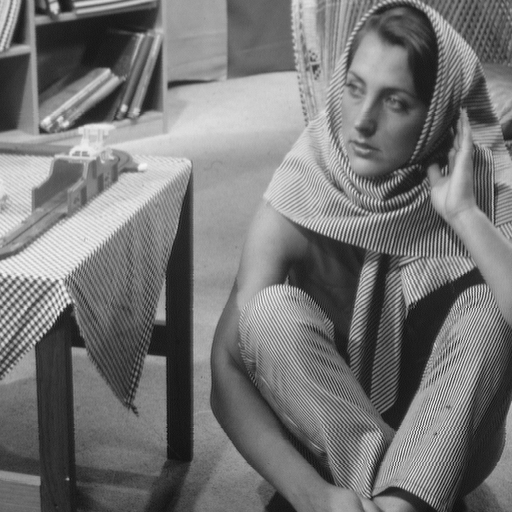}
\hfill
\includegraphics[width=0.3\textwidth]{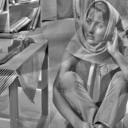}
\hfill
\includegraphics[width=0.3\textwidth]{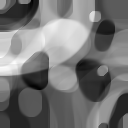}
\caption{From left to right: data $f$, indecomposable rest $g$, isolated TV-eigenfunctions\label{fig:barb}}
\end{figure}

\section{Applications}
\label{sec:applications}
In the following, we discuss different applications for the proposed spectral decomposition scheme \eqref{scheme} and extinction profiles of the gradient flow \eqref{gradflow}. After giving a short insight into the numerical computation of extinction profiles, we use the scheme and the gradient flow to compute and compare two different decompositions of an 1-dimensional signal into eigenfunctions of the total variation. Thereafter, we illustrate the use of the gradient flow and its extinction profiles for spectral graph clustering.
 
\subsection{Numerical Computation of Extinction Profiles}
\label{sec:numerics}
Here we describe how to calculate the extinction profiles $p_n^*$ of $f_n$ in scheme~\eqref{scheme}. Given a time step size $\delta>0$ the solution $u_k$ of the gradient flow \eqref{gradflow} at time $t_k=\delta k$ with datum $f$ is recursively approximated via the implicit scheme
$$u_k=\mathrm{arg}\min_{u\in\mathcal H}\frac{1}{2}\|u-u_{k-1}\|^2+\delta J(u),\quad k\in\mathbb{N},$$
where $u_0=f$. These minimization problems can be solved efficiently with a primal dual optimization algorithm \cite{chambolle2011first}. Defining $p_k:=(u_{k-1}-u_k)/\delta$ yields a sequence of pairs $(u_k,p_k)$ where $p_k\in\partial J(u_k)$ for every $k\in\mathbb{N}$. Hence, in order to compute an extinction profile of $f$, we keep track of the quantities $\hat{p}_k:=(p_k+p_{k-1})/2$ (cf.~the definition of $p^*$ in \eqref{extinction_profile}), and define the extinction profile $p^*$ of $f$ as the element $\hat{p}_k$ which has the highest Rayleigh quotient $R(\hat{p}_k):=\|\hat{p}_k\|^2/J(\hat{p}_k)$ before extinction of the flow. Note that since $\partial J(u)\subset\partial J(0)$ holds for all $u\in\mathcal H$ (cf.~\cite{bungert2019nonlinear}), both $p_k$ and, by convexity, also $\hat{p}_k$ are in particular elements of $\partial J(0)$ and, thus, have a Rayleigh quotient smaller or equal than one. Hence, an extinction profile, being even an eigenfunction, can be identified by having a quotient of one. To design a criterion for extinction of the flow we make use of the fact that the subgradients $p(t)$ in the gradient flow have monotonously decreasing norms and define the extinction time as the number $\delta k$ such that $\|p_k\|$ is below a certain threshold.

\subsection{1D Total Variation Example}
As already mentioned in Section~\ref{sec:scheme}, choosing $J$ to be the one-dimensional total variation yields a spectral case, i.e., the sequence $(f_n)$ in \eqref{scheme} weakly converges to zero which implies that $f=\sum_{i=0}^\infty c_i p_i^*$ is a decomposition into eigenfunctions. The rightmost images in Figure~\ref{fig:approx_of_data} show the data signal in red, whereas the other four images depict the approximation of $f$ by eigenfunctions of the gradient flow and the iterative scheme, respectively. The individual eigenfunctions (up to multiplicative constants) as computed by the gradient flow \eqref{gradflow} and the scheme \eqref{scheme} are given in Figure~\ref{fig:1d_efs}. Hence, the sum of the top four or the bottom 18 eigenfunctions, respectively, yield back the red signal from Figure~\ref{fig:approx_of_data}.

Note that the gradient flow only needs four eigenfunctions to generate the data and hence gives a very sparse representation. However, the individual eigenfunctions have decreasing spatial complexity. This is a fundamental difference to the system of eigenfunctions generated by \eqref{scheme} which -- being extinction profiles -- all have low complexity. This qualitative difference can also be observed in Figure~\ref{fig:approx_of_data} where the first eigenfunction of the gradient flow (top left) already contains all the structural information of the red signal whereas the approximation in the bottom row successively adds structure.

\begin{figure}[tbh]
\begin{center}
\includegraphics[width=1\textwidth,trim={1cm 8cm 0cm 8cm},clip]{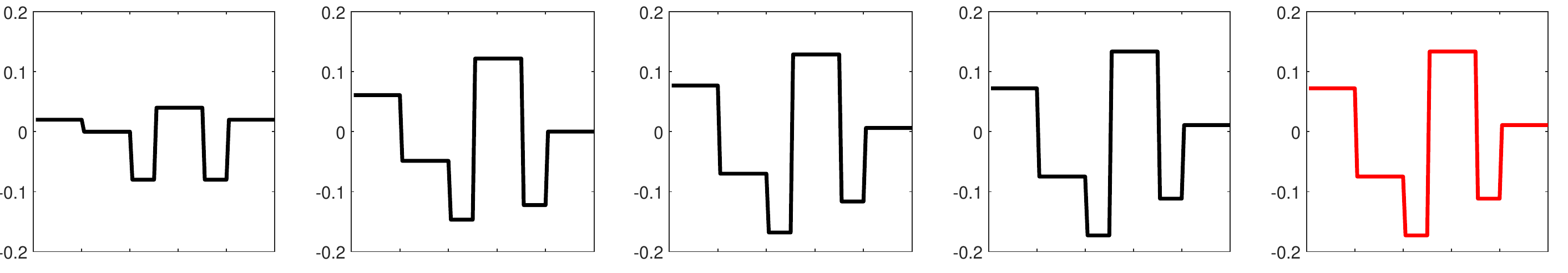}\\
\includegraphics[width=1\textwidth,trim={2cm 8cm 1cm 8cm},clip]{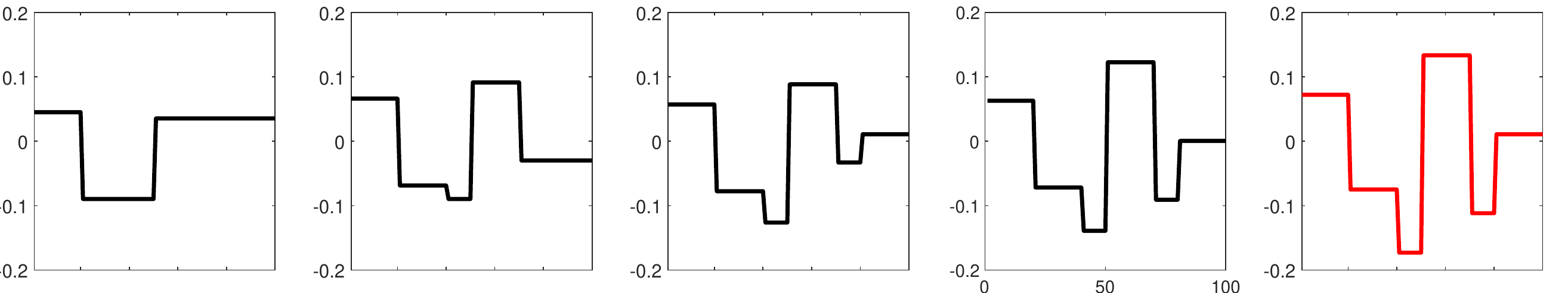}
\end{center}
\caption{Reconstruction of a 1D signal (red) based on TV eigenfunctions computed with the gradient flow scheme \eqref{gradflow} (top row) and the proposed extinction profile scheme \eqref{scheme} (bottom row). From left to right: sum of the first $1, 2, 3, 4$ (gradient flow), respectively $3, 6, 9, 18$ (recursive scheme) computed eigenfunctions, original data signal, see Figure~\ref{fig:1d_efs} for the individual eigenfunctions\label{fig:approx_of_data}}
\end{figure}
\begin{figure}[tbh]
\begin{center}
\includegraphics[width=1\textwidth,trim={1cm 6cm 0cm 7cm},clip]{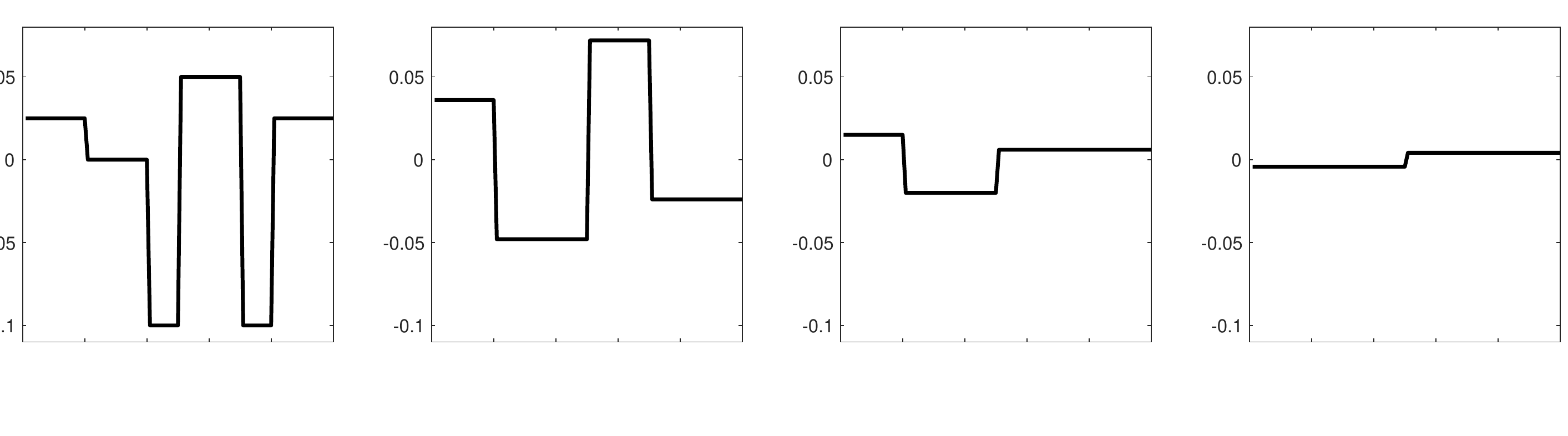}\\[-0.5cm]
\includegraphics[width=1\textwidth,trim={6cm 6cm 5.4cm 5cm},clip]{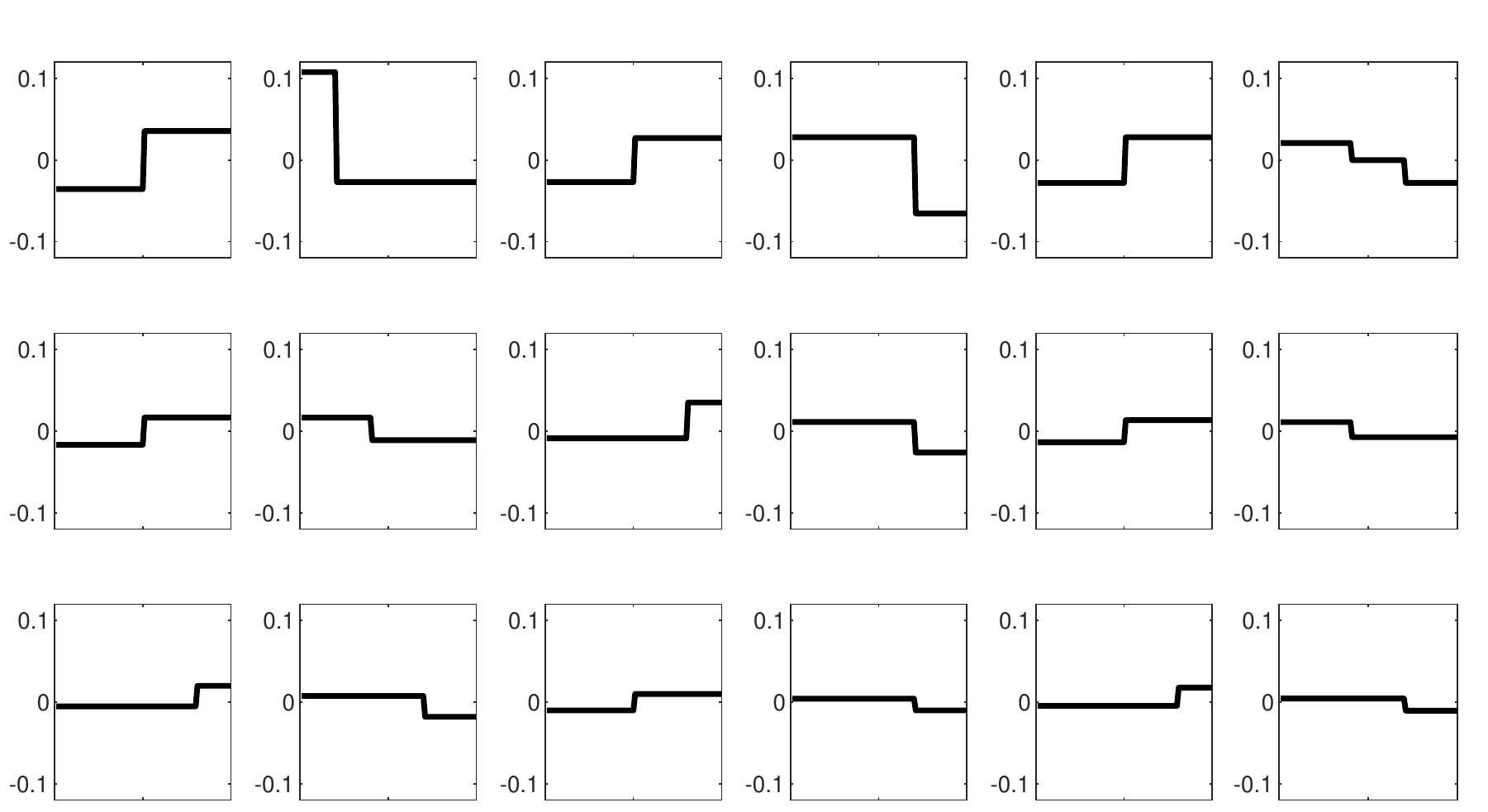}
\end{center}
\caption{Comparison between the computed 1D TV eigenfunctions of the gradient flow scheme \eqref{gradflow} and the proposed extinction profile scheme \eqref{scheme}. Top: All four eigenfunctions computed with \eqref{gradflow}. Bottom: first 18 eigenfunctions computed with \eqref{scheme}\label{fig:1d_efs}.}
\end{figure}
 
\subsection{Spectral Clustering with Extinction Profiles}
Spectral clustering arises in various real world applications, e.g., in discriminant analysis, machine learning, or computer vision. The aim in this task is to partition a given data set according to the spectral characteristics of an operator that captures the pairwise relationships between each data point. Based on the spectral decomposition of this operator one tries to find a partitioning of the data into sets of strongly related entities, called clusters, that should be clearly distinguishable with respect to a chosen feature. Within each cluster the belonging data points should be homogeneous with respect to this feature.

In order to model relationships between entities without further knowledge about the underlying data topology one may use finite weighted graphs. In this model each data point is represented by a vertex of the graph while the similarity between two data points is represented by a weighted edge connecting the respective vertices. 
For details on data analysis using finite weighted graphs we refer to \cite{Elmoataz_2015,Meng_2017}. In the literature it is well-known that there exists a strong mathematical relationship between spectral clustering and various minimum graph cut problems. For details we refer to \cite{Luxburg_2007}.

For the task of spectral clustering one is typically interested in the eigenvectors of a discrete linear operator known as the \textit{weighted graph Laplacian} $\Delta_w$, which can be represented as a matrix $L$ of the form $L=D-W$, for which $D$ is a diagonal matrix consisting of the degree of each graph vertex and $W$ is the adjacency matrix capturing the edge weights between vertices.
Determining the discrete spectral decomposition of the graph Laplacian $L$ is a common problem in mathematics and thus easy to compute. After determining the eigenvectors of $L$ one performs the actual clustering, e.g., via a standard $k$-means algorithm or simple thresholding. 
Note that in various applications a spectral clustering based on solely one eigenvector, i.e., the corresponding eigenvector of the second-smallest eigenvalue, already yields interesting results, e.g., for image segmentation \cite{Meng_2017}.
On the other hand, due to the linear nature of the graph Laplacian this approach is rather restricted in many real world applications. 
For this reason one aims to perform spectral clustering based on eigenfunctions of a nonlinear, possibly more suitable, operator. Bühler and Hein proposed in \cite{Buehler_2009} an iterative scheme to compute eigenfunctions of a nonlinear operator known as the \textit{weighted graph $p$-Laplacian}
\begin{equation} 
\label{eq:graph-p-laplacian}
\Delta_{w,p} f(x) \ = \ \sum_{y \sim x} w(x,y)^\frac{p}{2} |f(y) - f(x)|^{p-2} (f(y) - f(x)).
\end{equation}
Note that this operator is a direct generalization of the standard graph Laplacian $\Delta_w = \Delta_{w,2}$ for $p=2$.
Their idea consists in computing an eigenfunction of the linear Laplacian $\Delta_{w}$ and use this as initialization for a non-convex minimization problem of a Rayleigh quotient which leads to an eigenfunction of $\Delta_{w,p}$ with $p < 2$. This procedure is repeated iteratively for decreasing $p \rightarrow 1$ by using the intermediate solutions as initialization for the next step.
Their method already leads to satisfying results in situations in which a linear partitioning of the given data is not sufficient. However, as the authors state themselves, this approach often converges to unwanted local minima and is restricted to eigenfunctions corresponding to the second-smallest eigenvalue.

In Figure \ref{fig:tv-cluster} we compare the spectral clustering approach from \cite{Buehler_2009} based on the nonlinear graph $p$-Laplace operator with the extinction profiles introduced in Section~\ref{sec:gf_eigenfunctions}. For this we consider the following one-homogeneous, convex functional defined on vertex functions of a finite weighted graph
\begin{equation}
J(f) \ = \ \frac{1}{2} \sum_{x\in V} ||\nabla_w f(x)||_1.
\end{equation}
Here, $\nabla_w f(x)$ denotes the weighted gradient of the vertex function $f$ in a vertex~$x$. Note that this mimics a strong formulation of TV in the continuous case. The subgradient $\partial J$ corresponds to the graph $1$-Laplacian as special case of the graph $p$-Laplacian in \eqref{eq:graph-p-laplacian} for $p=1$.
We test the different approaches on the ``Two Moon'' dataset with low noise variance (top row) and a slightly increased noise variance (bottom row). The first column shows the computed nonlinear eigenfunction by the Bühler and Hein approach. In case of low noise variance (top) the eigenfunction takes only two values and is piece-wise constant, partitioning the data well. 
However, the eigenfunction takes more values in the noisy case (bottom) and a subsequent $k$-means-based clustering with $k=2$ does not yield a good partitioning of the data. In the center column we depict the extinction profiles computed with \eqref{gradflow}, initialized with random values on the graph vertices. Note that both eigenfunctions are piece-wise constant and take only two values, thus inducing a binary partitioning directly. However, similar to the Bühler-Hein eigenfunction, the found eigenfunction for the noisy case is not suitable to partition the dataset correctly. Hence, we performed a third experiment in the right column in which we initialized $5\%$ of the nodes per cluster with the values $\pm 1$, respectively and set the others to zero. Thus, we enforced the computation of eigenfunctions that correctly partition the data. This can be interpreted as a semi-supervised spectral clustering approach.
\begin{figure}[tbh]
\includegraphics[trim=50 12 35 60,clip,width=0.24\textwidth]{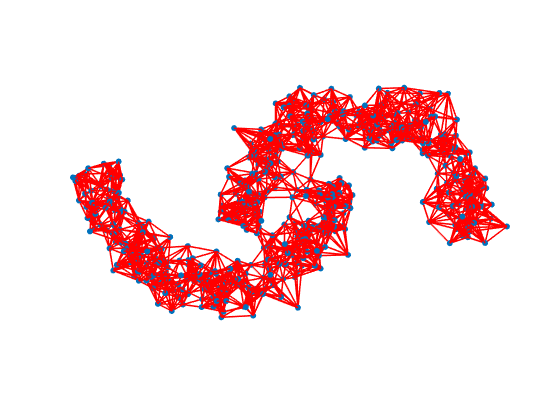}
\includegraphics[trim=50 12 35 60,clip,width=0.24\textwidth]{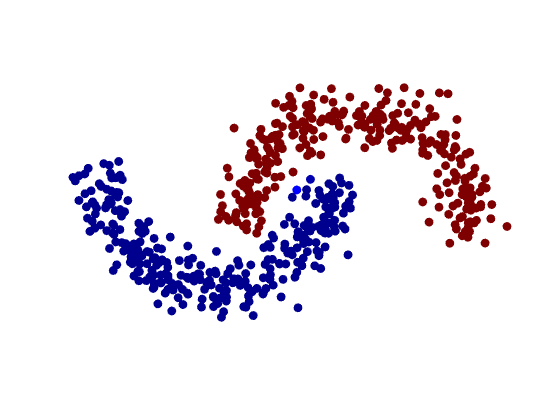}
\includegraphics[trim=50 12 35 60,clip,width=0.24\textwidth]{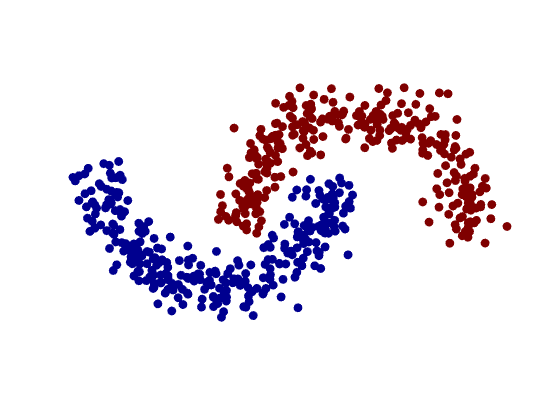}
\includegraphics[trim=50 12 35 60,clip,width=0.24\textwidth]{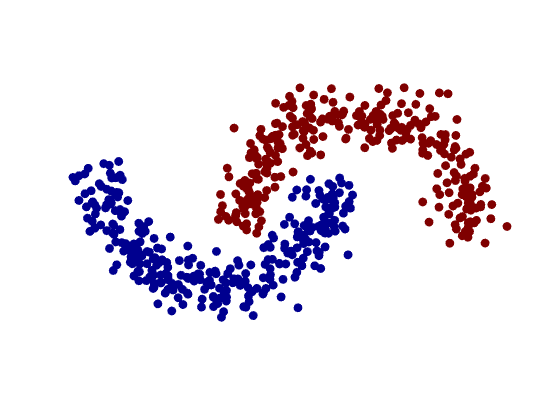}\\[-0.5cm]
\includegraphics[trim=50 12 35 60,clip,width=0.24\textwidth]{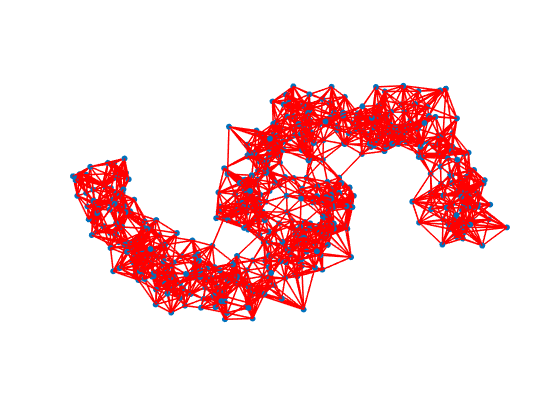}
\includegraphics[trim=50 12 35 60,clip,width=0.24\textwidth]{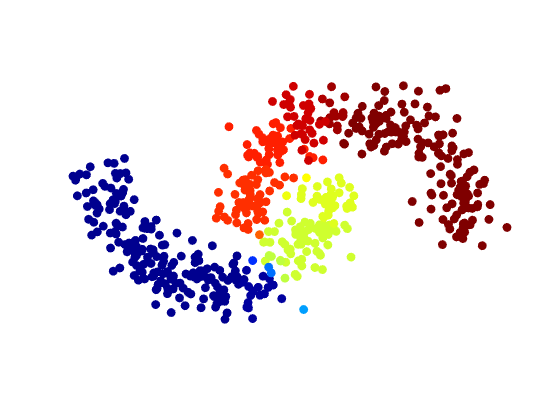}
\includegraphics[trim=50 12 35 60,clip,width=0.24\textwidth]{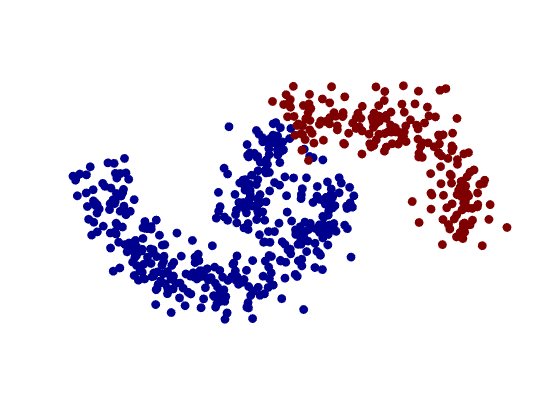}
\includegraphics[trim=50 12 35 60,clip,width=0.24\textwidth]{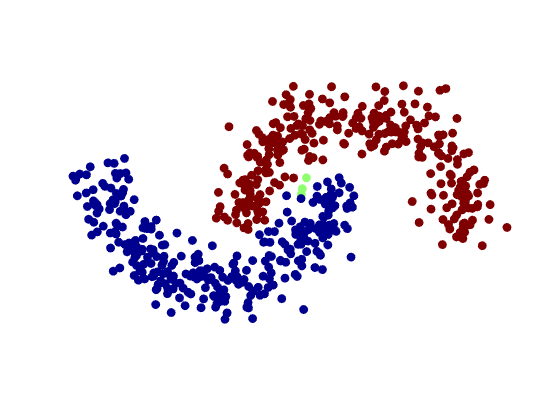}\\[-0.7cm]
\caption{Spectral clustering results on the ``Two Moon'' dataset based on the \textit{nonlinear graph $1$-Laplace operator} for two different levels of noise variance (top: $\sigma^2=0.015$, bottom: $\sigma^2=0.02$). Eigenfunctions computed from left to right: $k$-nearest neighbor graph for $k=10$, Bühler and Hein  approach \cite{Buehler_2009}, extinction profile with random initialization, extinction profile with $5\%$ manually labeled data points.\label{fig:tv-cluster}}
\end{figure}

In conclusion we state that spectral clustering based on nonlinear eigenfunctions is a potentially powerful tool for applications in data analysis and machine learning. However, we note that neither the Bühler and Hein approach discussed above nor extinction profiles guarantee a correct partioning of the data, in general. We could mitigate this drawback by using the fact that the chosen initialization of the gradient flow influences its extinction profile. 

\subsection{Outlook: Advanced Clustering with Higher-Order Eigenfunctions}
Finally, we demonstrate some preliminary results of our numerical experiments on a more challenging data set known as ``Three Moons''. In this case one requires for spectral clustering an eigenfunction that is constant on each of the three half-moons. Thus, we aim to find eigenfunctions of the graph $1$-Laplacian that correspond to a \emph{higher} eigenvalue than the second-smallest one. For this reason it is apparent that Bühler and Hein's method in its simplest variant (without subsequent splitting) always fails in this scenario (see top-right image in Figure~\ref{fig:three_moon}). Also extinction profiles, having the lowest possible eigenvalue of all subgradients $p(t)$ of the gradient flow \eqref{gradflow} lead to unreasonable results. However, one can still make use of the other subgradients and select those that are close to an eigenfunction, which can be measured by the Rayleigh quotient. The second row in Figure~\ref{fig:three_moon} shows three subgradients with a Rayleigh quotient of more than $90\%$, as they occur in the gradient flow. Note that the last one coincides with the extinction profile and obviously fails in separating all three moons since it only computes a binary clustering. Similarly, the first subgradient finds four clusters. The correct clustering into three moons is achieved by the subgradient in the center which was the last eigenfunction to appear before the extinction profile. This underlines the need of higher-order eigenfunctions for accurate multi-class spectral graph clustering.
\begin{figure}[tbh]
\centering
\includegraphics[trim=50 12 35 60,clip,width=0.32\textwidth]{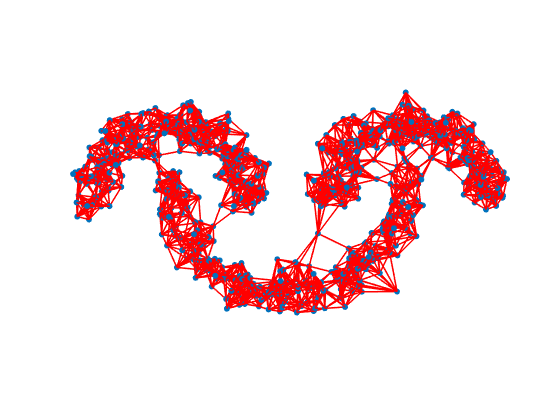}
\includegraphics[trim=50 12 35 60,clip,width=0.32\textwidth]{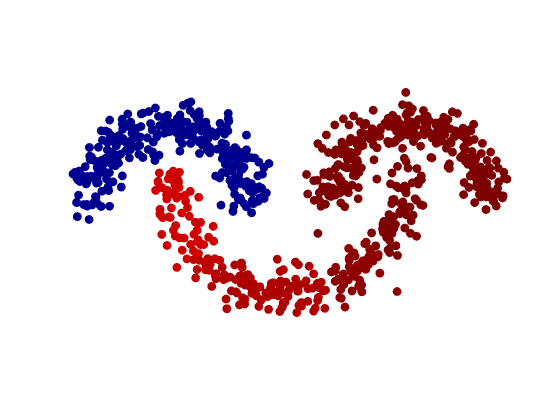}\\[-0.7cm]
\includegraphics[trim=50 12 35 60,clip,width=0.32\textwidth]{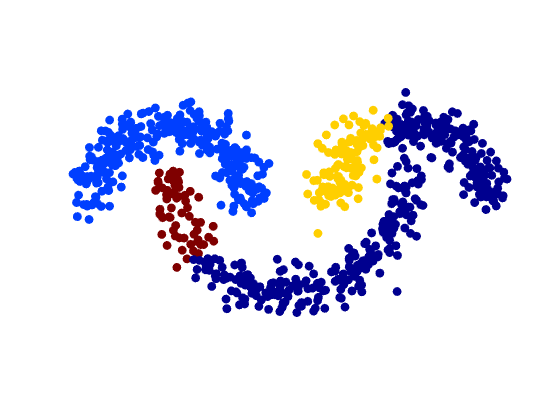}
\includegraphics[trim=50 12 35 60,clip,width=0.32\textwidth]{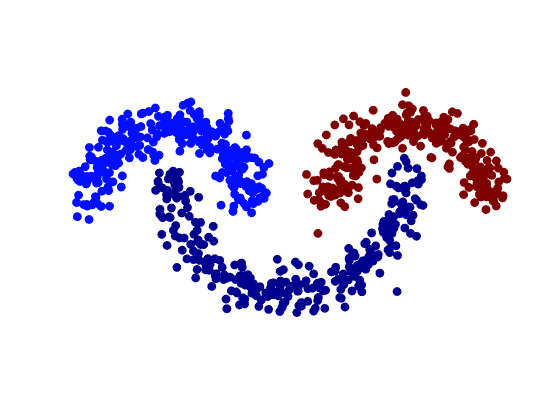}
\includegraphics[trim=50 12 35 60,clip,width=0.32\textwidth]{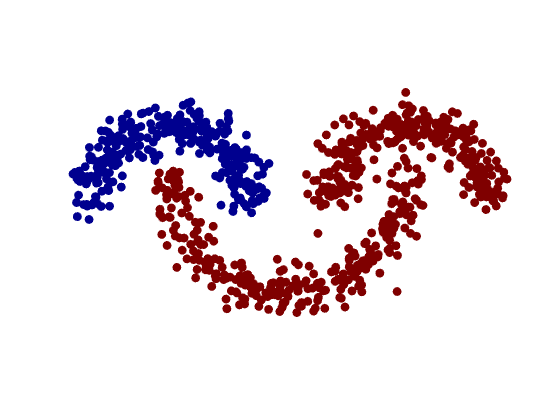}\\[-0.7cm]
\caption{Comparison between the computed eigenfunction of the Bühler and Hein approach (top) and three subgradients of the gradient flow with decreasing norm and Rayleigh quotient $\geq 90\%$ (bottom) on the ``Three moons'' dataset.\label{fig:three_moon}}
\end{figure}

\section*{Acknowledgments}
This work was supported by the European Union’s Horizon 2020 research and innovation programme under the Marie Skłodowska-Curie grant agreement No 777826 (NoMADS). LB and MB acknowledge further support by ERC via Grant EU FP7 – ERC Consolidator Grant 615216 LifeInverse.

\clearpage
\bibliography{bib}{}
\bibliographystyle{plain}
\end{document}